\documentclass[submission]{eptcs}
\providecommand{\event}{ACT 2022} 
\usepackage{underscore}           
\usepackage[utf8]{inputenc}
\usepackage[english]{babel}
\usepackage{amsmath,amssymb,amsthm}
\usepackage{relsize}
\usepackage{graphicx}
\usepackage{float}
\usepackage{xcolor}
\usepackage{mathrsfs}
\usepackage{enumitem}
\usepackage{mathdots}
\usepackage{subcaption}
\usepackage{tikz}
\usepackage{tikz-cd}
\usetikzlibrary{babel}
\usetikzlibrary{arrows.meta}
\usetikzlibrary{decorations.pathmorphing}

\title{Lax Liftings and Lax Distributive Laws}
\author{Ezra Schoen
\institute{University of Strathclyde\\Glasgow, Scotland}
}
\def\titlerunning{Lax Liftings and Lax Distributive Laws}
\def\authorrunning{E. Schoen}

\let\temp\phi
\let\phi\varphi
\let\varphi\temp
\let\temp\epsilon
\let\epsilon\varepsilon
\let\varepsilon\temp

\DeclareMathOperator{\Rel}{\mathbf{Rel}}
\DeclareMathOperator{\Sets}{\mathbf{Sets}}
\DeclareMathOperator{\N}{\mathcal{N}}
\DeclareMathOperator{\M}{\mathcal{M}}

\DeclareMathOperator{\Hom}{\mathbf{Hom}}
\DeclareMathOperator{\relto}{\multimap}
\DeclareMathOperator{\gr}{gr}
\DeclareMathOperator{\id}{id}
\DeclareMathOperator{\Lift}{\mathbf{Lift}}

\newcommand{\op}{\text{op}}
\newcommand{\conv}{\circ}
\newcommand{\sqto}{\rightsquigarrow}

\newcommand{\smalltilde}{\scalebox{0.7}{$\sim$}}

\newcommand{\twiddle}[1]{{#1}^{\text{\smalltilde}}}

\newcounter{mcount}

\theoremstyle{plain}
\newtheorem{lemma}[mcount]{Lemma}

\newtheorem{prop}[mcount]{Proposition}
\newtheorem{theorem}[mcount]{Theorem}
\newtheorem{corollary}[mcount]{Corollary}

\theoremstyle{definition}
\newtheorem{defi}[mcount]{Definition}

\theoremstyle{remark}
\newtheorem{remark}[mcount]{Remark}
\newtheorem{example}[mcount]{Example}

\newtheorem{conj}[mcount]{Conjecture}

\begin{document}

\maketitle

\begin{abstract}
Liftings of endofunctors on sets to endofunctors on relations are commonly used to capture bisimulation of coalgebras. Lax versions have been used in those cases where strict lifting fails to capture bisimilarity, as well as in modeling other notions of simulation. This paper provides tools for defining and manipulating lax liftings. 

As a central result, we define a notion of a lax distributive law of a functor over the powerset monad, and show that there is an isomorphism between the lattice of lax liftings and the lattice of lax distributive laws.

We also study two functors in detail: (i) we show that the lifting for monotone bisimilarity is the minimal lifting for the monotone neighbourhood functor, and (ii) we show that the lattice of liftings for the (ordinary) neighbourhood functor is isomorphic to P(4), the powerset of a 4-element set. 
\end{abstract}

\section{Introduction}

Coalgebras for an endofunctor are a general model of state-based transition systems. \cite{Jacobs2016} Bisimulations are a central concept in the study of coalgebras, describing behavioral equivalence of states. Going back to \cite{Rutten1998}, bisimulations of $F$-coalgebras in $\Sets$ have been defined as prefixed points of $\bar F$, the extension of $F$ to $\Rel$, the category of sets and relations. 

One issue is that $\Rel$ places high demands on extensions: if $\widetilde F:\Rel\to\Rel$ is to be a strict functor that preserves the ordering of relations, and coincides with $F$ on graphs of functions, then $\widetilde F$ only exists if $F$ preserves weak pullbacks\cite{Carboni1991}; and if $F$ preserves them, it is unique \cite{Bird1997} and equal to the Barr lifting $\bar F$. \cite{Barr1970} This situation is undesireable for two reasons:
\begin{itemize}
\item The elegant extension-based framework for bisimulation cannot be directly applied to coalgebras of type $F$ when $F$ does not preserve weak pullbacks. Neighbourhood-type functors are the most prominent example of such $F$.
\item While the lifting $\bar F$ can be used to reason about bisimulation, other notions of simulation or equivalence of coalgebras cannot be expressed in the same way, since there are no other strict extensions. 
\end{itemize}

To remedy this, various weaker notions of extension have been proposed.\cite{Thijs1996, Baltag2000, Hughes2004, Marti2015}. Finding explicit examples has proceeded in a mostly ad-hoc fashion. The aim of this paper is to provide tools to reason about lax lifting in a more principled way. This paper is based on chapter 3 of the author's MSc thesis \cite{Schoen2021}.

\vspace{10pt}

Our main contribution is a new notion of a \emph{lax distributive law}, which we will show are in one-to-one correspondence with lax liftings. Distributive laws at their most general are simply natural transformations $FG\Rightarrow GF$ for two functors $F,G$. In most cases however, at least one of the two functors $F$ and $G$ is taken to be a monad, and the distributive law is required to interact `nicely' with the monad structure. 

The connection between liftings and distributive laws originates in \cite{Beck1969}, which focused on monad-monad interactions. Mulry \cite{Mulry1994} proved the equivalence between distributive laws of a functor $F$ over a monad $T$ and liftings of $F$ to the Kleisli category of $T$. 

More recently, some notions of `weak distributive law' have been studied \cite{Street2009}; these, like Beck, pertain to monad-monad interaction, and involve weakening some of the conditions on Becks original distributive laws. Closer to the work in this paper are the lax distributive laws in \cite{Tholen2016}, though again these focus on monad-monad interactions. 

Aside from their connection to monads, distributive laws are of interest in their own right. They feature centrally in the bialgebraic approach to operational semantics. \cite{Turi1997, Klin2006} In the theory of automata, morphisms of distributive laws can provide various determinization procedures. \cite{Zetzsche2021}

\vspace{10pt}

We also analyse the liftings for two specific functors in detail:
\begin{itemize}
\item We prove that the minimal lifting for the monotone neighbourhood functor is given by the lifting $\widetilde \M$. This lifting has previously been used \cite{Santocanale2010}; our result shows that $\widetilde\M$ is in some sense universal for $\M$.
\item We give a complete description of the liftings for the ordinary neighbourhood functor. Equivalence notions between neighbourhood structures can be quite complex. \cite{Hansen2009} The classification in this paper shows that any notion of bisimulation between neighbourhood structures based on lax liftings will be almost trivial, since none of the 16 possible liftings make meaningful use of the input relation.
\end{itemize}

\emph{Outline}

\vspace{10pt}

In section 2, we show that for a fixed functor, the lax liftings form a complete lattice. This implies that any functor admits a minimal, ``maximally expressive" lifting. We show that for weak pullback-preserving functors, the minimal lifting coincides with the Barr lifting. 

In section 3, we define lax distributive laws, and show that there is an isomorphism between the lattice of lax liftings, and the lattice of lax distributive laws. We also characterize those distributive laws that correspond to liftings that are symmetric and 
diagonal-preserving. 

In section 4, we study the monotone and ordinary neighbourhood functors in more detail. For the monotone neighbourhood functor, we show that the known lifting $\widetilde \M$ is minimal. For the ordinary neighbourhood functor, we show that the lattice of liftings is isomorphic to $P(4)$ by giving an explicit description of all 16 liftings. 

\section{Preliminaries and basic properties}

\begin{defi}
We write $\Rel$ for the category of sets and relations. The objects of $\Rel$ are sets, and a morphism $R\in \Hom_{\Rel}(X,Y)$ is given by a subset $R\subseteq X\times Y$.

Given two relations $R:X\relto Y$ and $S:Y\relto Z$, we write $R;S:X\relto Z$ for their composition $R;S = \{(x,z)\in X\times Z\mid \exists y:xRySz\}$. Note that the order of composition is reversed from function composition. 

Given a relation $R:X\relto Y$, we write $R^\conv$ for its converse; that is,
\[
R^\conv = \{(y,x)\mid (x,y)\in R\}
\]

Given a function $f:X\to Y$, we write $\gr(f)$ for its \emph{graph}, which is the relation
\[
\gr(f) = \{(x,y)\mid f(x) = y\}
\]

The category $\Rel$ is enriched over posets, where relations are ordered by inclusion. This makes $\Rel$ into a 2-category (in fact, it is the canonical example of an allegory). The operation $(-)^\conv$ is the morphism part of a functor $(-)^\conv:\Rel\to \Rel^\op$, which is an isomorphism of 2-categories. We write $\gr^\conv:\Sets^\op\to \Rel$ for the composition $(-)^\conv\circ \gr$. 
\end{defi}

\begin{remark}
The category $\Rel$ is isomorphic to the Kleisli category for the powerset monad. The assignment $f\mapsto \gr(f)$ is the morphism part of the left adjoint $\gr$ in the free-forgetful adjunction $\gr\dashv P$ that arises out of the Kleisli category construction. For a given relation $R:X\relto Y$, we will write $\chi_R:X\to PY$ for the corresponding Kleisli morphism. Conversely, for a Kleisli morphism $f:X\to PY$, we will write $\lfloor f\rfloor:X\relto Y$ for the corresponding relation. 

The converse of a Kleisli morphism $f:X\to PY$ will be written as
\[
f^\flat:Y\to PX: y\mapsto \{x\mid f(x)\ni y\}
\]
\end{remark}

\begin{defi}
Let $F:\Sets\to\Sets$ be a functor. A \emph{(lax) $F$-lifting} is a lax 2-functor $L:\Rel\to\Rel$ such that 
\[
\begin{tikzcd}
\Rel\rar{L}\arrow[dr, "\geq" marking, phantom] & \Rel & \Rel\rar{L}\arrow[dr, "\geq" marking, phantom] & \Rel\\
\Sets\uar{\gr} \rar{F} & \Sets\uar{\gr}& \Sets^\op\uar{\gr^\conv}\rar{F^\op} & \Sets^\op\uar{\gr^\conv}
\end{tikzcd}
\] commute up to the indicated inequalities. A lifting $L$ is called \emph{symmetric} if 
\[
\begin{tikzcd}
\Rel^\op\rar{L^\op} & \Rel^\op\\
\Rel\uar{(-)^\conv}\rar{L} & \Rel\uar{(-)^\conv}
\end{tikzcd}
\]
commutes; it is called \emph{diagonal-preserving} if it strictly preserves identities. 

Explicitly, we can expand the above into the following 5 conditions:
\begin{description}
\item[1. (2-cells)] For all $R, S:X\relto Y$, if $R\leq S$, then $LR \leq LS$.
\item[2. (lax functoriality)] For all $R: X\relto Y$ and $S:Y\relto Z$, we have $LR;LS \leq L(R;S)$. 
\item[3. (lifting)] For all $f:X\to Y$, we have
\[
\gr(Ff)\leq L\gr(f), \qquad \gr^\conv(Ff)\leq (L\gr(f))^\conv
\]
\item[4. (diagonal-preserving)] For all $X$, we have $L\Delta_X \leq \Delta_{FX}$.
\item[5. (symmetry)] For all $R:X\relto Y$, we have
\[
L(R^\conv) = (LR)^\conv
\]
\end{description}
\end{defi}

\begin{remark}
The above includes various notions of lifting that have been previously been studied. Some authors (e.g. \cite{Kurz2016}) have taken ``lifting'' to be synonymous with the Barr lifting (see below). The notion of ``(weak) relator" in \cite{Baltag2000} and \cite{Thijs1996} strengthen condition 2 to strict functoriality (although \cite{Thijs1996} does not require monotonicity). The notion used in \cite{Marti2015} is almost identical, the only difference being that they require symmetry. 
\end{remark}

We give some examples:
\begin{example}
\begin{enumerate}[label = (\arabic{mcount}.\arabic*)]
\item For all functors $F:\Sets\to\Sets$, there is the lifting $F_\top:\Rel\to\Rel$ given by
\[
F_\top(R:X\relto Y) = FX\times FY
\]
This lifting is symmetric, but does not preserve diagonals unless $|FX| \leq 1$ for all $X$.
\item Any relation $R:X\relto Y$ is presented as a span $R = \gr^\conv(\pi_1^R);\gr(\pi_2^R)$ by the two projection functions $\pi_1^R:R\to X$ and $\pi_2^R:R\to Y$. This motivates the definition
\[
\bar FX = \gr^\conv(F\pi_1);\gr(F\pi_2)
\]
$\bar F$ is known as the \emph{Barr lifting}; it originates in \cite{Barr1970}. In general, $\bar F$ is not lax but oplax, meaning $LR;LS \geq L(R;S)$. However, if $F$ preserves weak pullbacks, then $\bar F$ is a strict functor which strictly preserves graphs and converse graphs.\cite{Kurz2016} Since the diagonal is the graph of the identity, $\bar F$ also preserves diagonals. 
\item \label{ex:monotonelift} The \emph{Neighborhood functor} is defined to be the functor $\N = P P$. The action on a morphism $f: X\to Y$ is given by
\[
(\N f)U = \{v\mid f^{-1}(v)\in U\}
\]
The \emph{Monotone neighborhood functor} is the subfunctor $\M$ of $\N$ defined by
\[
\M X = \{U\in \N X\mid u\in U\text{ and }u \subseteq u'\implies u'\in U\}
\]
One lifting for the monotone neighbood functor is given by
\begin{align*}
\widetilde\M(R:X\relto Y) = \{(U, V)\mid &\forall u\in U\exists v\in V:\forall y\in v\exists x\in u:xRy\\
\text{ and }&\forall v\in V\exists u\in U:\forall x\in u\exists y\in v: xRy\}
\end{align*}
This lifting originates in \cite{Santocanale2010} where it was used to prove uniform interpolation for monotone modal logic. A closely related notion of bisimulation appeared earlier in \cite{Hansen2004}.
\end{enumerate}
\end{example}

We also state a simple lemma on lax liftings:

\begin{lemma}\label{lemma:cospan}
Let $L$ be an $F$-lifting. For all relations $R:X\relto Y$ and all functions $f:X'\to X$ and $g:Y'\to Y$, we have
\[
L(\gr(f);R;\gr^\conv(g)) = \gr(Ff);LR;\gr^\conv(Fg)
\]
\end{lemma}

This is lemma 3.10(iii) in \cite{Schoen2021}.

\vspace{10pt}

For a given functor $F:\Sets\to\Sets$, write $\Lift(F) = \{L:\Rel\to\Rel\mid L\text{ is an $F$-lifting}\}$. Liftings are naturally ordered pointwise: we say $L\leq L'$ if and only if for all $R$, we have $LR \leq L'R$. 

\begin{theorem}\label{thm:meets}
Fix a functor $F:\Sets\to\Sets$. The class $\Lift(F)$ forms a complete lattice, with meets given by
\[
\left(\bigwedge_{i\in I} L_i\right)R := \bigcap_{i\in I}(L_iR)
\]
\end{theorem}

\begin{proof}
See appendix.
\end{proof}

Since complete lattices have a minimal element, we get the following corollary:

\begin{corollary}
Every endofunctor on $\Sets$ admits a minimal lifting.
\end{corollary}

The significance of this corollary is the following: each lifting gives rise to a corresponding notion of simulation of coalgebras, as well as a modal logic. If for two liftings $L, L'$ we have $L\leq L'$, then $L$-simulation distinguishes more states than $L'$-simulation, and $L$-logic is more expressive than $L'$-logic. A minimal lifting hence induces a maximally discerning notion of (bi)simulation, and a maximally expressive logic (among those that arise from lax liftings). \cite{Schoen2021}

\vspace{10pt}

In case $F$ is weak pullback-preserving, we have an explicit description of its minimal lifting.
\begin{prop}
Let $F:\Sets\to\Sets$ be weak pullback-preserving. Then $\bar F$ is minimal among the $F$-liftings.
\end{prop}

\begin{proof}
Let $L$ be a lifting for $F$. Then let $R:X\relto Y$ be a relation. We know that $R$ is presented as a span $R = \gr^\circ(\pi^R_X);\gr(\pi^R_Y)$ with $\pi^R_X:R\to X$ and $\pi^R_Y:R\to Y$ being the projection functions. So,
\[
LR = L(\gr^\circ(\pi^R_X);\gr(\pi^R_Y))\geq L(\gr^\circ(\pi^R_X));L(\gr(\pi^R_Y))\geq \gr^\circ(F\pi^R_X);\gr(F\pi^R_Y) = \bar FR
\]
\end{proof}

There is also a natural involution on liftings, induced by $(-)^\conv$:
\begin{defi}
For an $F$-lifting $L$, we define the lifting $\twiddle L$ as
\[
\twiddle L(R) := (L(R^\conv))^\conv
\]
\end{defi}

It is simple to prove that $\twiddle L$ is a lifting when $L$ is. \cite{Schoen2021}

Natural transformations between functors also induce a map between the associated liftings:
\begin{theorem}\label{thm:transform}
Let $F,G:\Sets\to\Sets$ be functors, and let $\eta:F\Rightarrow G$ be a natural transformation.
\begin{enumerate}[label = (\roman*)]
\item For every $G$-lifting $L$, the assignment
\[
R\mapsto \{(x,y)\in FX\times FY \mid (\eta(x),\eta(y))\in LR\}
\]
constitutes an $F$-lifting $\eta^*L$.
\item $\eta^*$ preserves arbitrary meets and $\twiddle{(-)}$. 
\item If $L$ is symmetric, so is $\eta^*L$.
\item If $\eta$ is everywhere injective, then if $L$ preserves diagonals, so does $\eta^*L$. 
\end{enumerate}
\end{theorem}

Note that joins are not preserved in general: in particular, the minimal lifting is rarely preserved by $\eta^*$.

\begin{proof}
See appendix.
\end{proof}

From point (iv), together with the fact that the Barr lifting always preserves diagonals, we immediately get the following result:

\begin{corollary}
All subfunctors of a weak pullback-preserving functor admit a diagonal-preserving lifting. 
\end{corollary}

This motivates the following conjecture:

\begin{conj}\label{conj:1}
	The converse of the above: if $F$ has a diagonal-preserving lifting, it can be embedded in a weak pullback-preserving functor. 
\end{conj}

\section{Lax distributive laws}

In this section, we give an alternative characterization of relation lifting in terms of distributive laws. We will write $\mu:P^2\to P$ and $\eta:\id\to P$ for respectively the multiplication and unit of the powerset monad.

\vspace{10pt}

\begin{defi}
Let $F:\Sets\to\Sets$ be any functor. A \emph{lax distributive law for $F$} is a collection of maps $\lambda_-:FP(-)\to PF(-)$, satisfying:
\begin{description}
\item[(Monotonicity)] For any two functions $f,g:X\to PY$, if $f\leq g$, then
\[
\lambda_Y\circ Ff\leq \lambda_Y\circ Fg
\]
\item[(Lax naturality)] For any function $f:X\to PY$, we have
\[
PFf\circ\lambda_X\leq \lambda_{PY} \circ FPf
\]
\item[(Lax Eilenberg-Moore)] For any $Z$, we have
\[
\mu_{FZ}\circ P\lambda_Z\circ \lambda_{PZ} \leq \lambda_Z\circ F\mu_Z \text{ and }\lambda_Z\circ F\eta_Z \geq \eta_{FZ}
\]
\end{description}
There are also the optional properties
\begin{description}
\item[(Lax extensionality)] For any $Z$,
\[
\lambda_Z\circ F\eta_Z\leq \eta_{FZ}
\]
\item[(Symmetry)] For any map $f:X\to PY$, 
\[
(\lambda_Y\circ Ff)^\flat = \lambda_X\circ F(f^\flat)
\]
\end{description}
\end{defi}

\begin{defi}
Let $\lambda:FP\sqto PF$ be a lax distributive law. For a given relation $R:X\relto Y$, we define $L^\lambda R$ as $
L^\lambda R := \lfloor \lambda_Y\circ F\chi_R\rfloor$.

Conversely, for a lax lifting $L$ of $F$, we define $\lambda^L:FP\sqto PF$ as
$
\lambda^L := \chi_{L\ni}
$.
\end{defi}

The main theorem of this section states that these operations describe a bijective correspondence between lax liftings and lax distributive laws.

\begin{theorem}
Let $F:\Sets\to\Sets$ be a functor.
\begin{enumerate}[label = (\roman*)]
\item The operations $L\mapsto \lambda^{L}$ and $\lambda\mapsto L^\lambda$ are inverse to each other.
\item If $L$ is a $F$-lifting, then $\lambda^L$ is a lax distributive law. Moreover, if $L$ preserves diagonals then $\lambda^L$ is laxly extensional, and if $L$ is symmetric, then $\lambda^L$ is symmetric.
\item If $\lambda$ is a lax distributive law, then $L^\lambda$ is a $F$-lifting. Moreover, if $\lambda$ is laxly extensional, then $L^\lambda$ preserves diagonals, and if $\lambda$ is symmetric, then $L^\lambda$ is symmetric.
\end{enumerate}
\end{theorem}

\begin{proof}
\begin{enumerate}[label = (\roman*)]
\item We calculate
\begin{align*}
\lfloor\lambda^{L^\lambda}_Z\rfloor &= L^\lambda(\ni_Z)= \lfloor\lambda_Z\circ F\chi_{\ni}\rfloor = \lfloor \lambda_Z\circ F\id_{PZ}\rfloor =  \lfloor\lambda_Z\rfloor
\end{align*}
showing $\lambda^{L^\lambda} = \lambda$.

For the other equality, we get
\begin{align*}
L^{\lambda^L}(R) &= \lfloor\lambda^L\circ F\chi_R\rfloor\\
&= \lfloor \chi_{L\ni}\circ F\chi_R\rfloor\\
&\overset{*}{=} \lfloor \mu\circ P\chi_{L\ni}\circ \eta\circ F\chi_R\rfloor  \\
&\overset{**}{=} \lfloor \eta\circ F\chi_R\rfloor;\lfloor \chi_{L{\ni}}\rfloor  \\
&= \gr(F\chi_R);L\ni\\
&= L(\gr(\chi_R);\ni)& \text{ by lemma }\ref{lemma:cospan}\\
&\overset{***}{=} LR
\end{align*}
where in (*), we use one of the unit laws for monads, in (**) we use that $\lfloor -\rfloor$ turns Kleisli composition into relational composition, and in (***) we use the (easily verified) identity $\gr(\chi_R);{\ni} = R$. 
\item We check the conditions in order.
\begin{description}
\item[(Monotonicity)] We see that
\[
\lfloor \lambda^L_Y\circ Ff\rfloor = \lfloor \lambda^L_Y\circ F\chi_{\lfloor f\rfloor}\rfloor = L^{\lambda^L}(\lfloor f\rfloor) = L(\lfloor f\rfloor)
\]
where we use point (i) for the final equality. Now monotonicity of $\lambda^L$ follows immediately from monotonicity of $L$.
\item[(Lax naturality)] Note that
\[
\left(\ni_X;\gr(f)\right)\subseteq \left(\gr(Pf);\ni_{PY}\right)
\]
since if $A\ni x$, then $Pf[A]\ni f(x)$. 

Now we see that
\begin{align*}
a\in PFf\circ \lambda_X^L(\Phi) &\iff \exists a': a = Ff(a')\text{ and }a'\in \lambda_X^L(\Phi)\\
&\iff \exists a':a = Ff(a')\text{ and }a'\in \chi_{L\ni}(\Phi)\\
&\iff \exists a': a = Ff(a')\text{ and }(\Phi,a')\in L({\ni_X})\\
&\iff (\Phi,a)\in L({\ni_X});\gr(Ff)\\
&\implies (\Phi,a)\in L({\ni_X};\gr(f))\\
&\implies (\Phi,a)\in L(\gr(Pf);\ni_{PY})\\
&\iff (\Phi,a)\in \gr(FPf);L({\ni_{PY}})&\text{by lemma }\ref{lemma:cospan}\\
&\iff a\in \lambda_{PY}^L \circ FPf(\Phi)
\end{align*}
\item[(Lax Eilenberg-Moore)] First, we write out that
\[
\mu\circ P\lambda^L_Z\circ \lambda^L_{PZ} = \mu \circ P(\chi_{L\ni})\circ \chi_{L\ni} = \chi_{L\ni;L\ni}
\]
since $\chi_{-}$ turns relational composition $;$ into Kleisli composition. Next, note that
\[
\gr(\mu);\ni = \ni;\ni
\]
since
\[
\bigcup_{A\in \mathcal{A}}A \ni x \text{ if and only if }\exists A: \mathcal{A}\ni A\text{ and }A\ni x
\]
So, we conclude that
\begin{align*}
\lfloor\mu\circ P\lambda^L_Z\circ \lambda^L_{PZ}\rfloor &= L\ni;L\ni \\
&\leq L(\ni;\ni) \\
&= L(\gr(\mu);\ni)\\
&= \gr(F\mu);L\ni&\text{ by lemma }\ref{lemma:cospan}\\
&= \lfloor \eta\circ F\mu\rfloor ; \lfloor \chi_{L\ni}\rfloor\\
&\overset{*}{=} \lfloor \mu \circ P\chi_{L\ni}\circ \eta\circ F\mu\rfloor\\
&= \lfloor \chi_{L\ni}\circ F\mu\rfloor\\
&= \lfloor\lambda_{Z}^L\circ F\mu\rfloor
\end{align*}
giving the first inequality; where in the equality (*) we use that $\lfloor -\rfloor$ turns relational composition into Kleisli composition. 

For the second inequality, we simply note that $\lfloor\eta_Z\rfloor = \gr(\id_Z)$, and so
\[
\lfloor\lambda^L_Z\circ F\eta_Z\rfloor = L^{\lambda^L}\lfloor\eta_Z\rfloor = L\gr(\id_Z) \geq \gr(F\id_Z) = \gr(\id_{FZ}) = \lfloor\eta_{FZ}\rfloor
\]
\item[(Lax extensionality)] Assume that $L$ is diagonal-preserving. We aim to show that $\lambda_L$ is laxly extensional. This follows simply from
\[
\lfloor\lambda^L_Z \circ F\eta_Z\rfloor = L\gr(\id_Z) \leq \gr(\id_{FZ}) = \lfloor\eta_{FZ}\rfloor
\]
\item[(Symmetry)] If $L$ is symmetrical, we get simply
\[
\lfloor(\lambda^L_Y\circ Ff)^\flat\rfloor = (L\lfloor f\rfloor)^\circ = L(\lfloor f\rfloor^\circ) = \lfloor\lambda_X\circ F(f^\flat)\rfloor
\]
\end{description}
\item We prove each of the five conditions.
\begin{description}
\item[(2-cells)] If $S\leq R$, then
\[
L^\lambda S = \lfloor \lambda_Y\circ F\chi_S\rfloor \leq \lfloor\lambda_Y\circ F\chi_R\rfloor = L^\lambda R
\]
by monotonicity of $\lambda$.
\item[(lax functoriality)] Let $R:X\relto Y$ and $S:Y\relto Z$. We draw the following diagram:
\[
\begin{tikzcd}[column sep = large, row sep = large]
FX\arrow[r, "F(\chi_{R;S})"] \arrow[r,""{name = U, below},phantom]\dar{F\chi_{R}} & FPZ\rar{\lambda_Z} & PFZ\\
FPY\arrow[r, "FP\chi_S"{name=D}]\dar{\lambda_Y} & FPPZ\uar{F\mu_Z}\dar{\lambda_{PZ}}\arrow[r,"P\lambda_Z\circ \lambda_{PZ}"]\arrow[r,""{name = W, below}.phantom] & PPFZ\arrow[u, "\mu_{FZ}"] \arrow[ul, "\geq" marking, phantom]\\
PFY\rar{PF\chi_{S}}\arrow[ur, "\leq" marking, phantom] & PFPZ\arrow[r,"P\lambda_Z"{name = V}] & PPFZ\arrow[u, equal]
\arrow[phantom, "=" marking, from=U, to=D]
\arrow[phantom, "=" marking, from=V, to=W]
\end{tikzcd}
\]
The top left square is $F$ applied to the Kleisli composite $\chi_{R;S}$. The top right square is lax Eilenberg-Moore, and the bottom left square is lax naturality. The bottom right square is a simple equality. 

The above diagram shows that
\[
L^\lambda R;L^\lambda S = \lfloor\mu_{FZ}\circ P(\lambda_Z\circ F\chi_{S})\circ \lambda_Y\circ F\chi_{R} \rfloor \leq \lfloor\lambda_Z\circ F(\chi_{R;S}) \rfloor= L^\lambda(R;S)
\]
as desired.
\item[(lifting)] Let $f:X\to Y$ be a morphism. Then 
\begin{align*}
L^\lambda \gr(f) &= L^\lambda(\lfloor\eta_Y\circ f\rfloor)= \lfloor\lambda_Y\circ F(\eta_Y\circ f)\rfloor= \lfloor\lambda_Y\circ F\eta_Y\circ Ff\rfloor\geq \lfloor\eta_{TY}\circ Ff\rfloor = \gr(Ff)
\end{align*}
by lax Eilenberg-Moore. We also have
\begin{align*}
\gr(Ff);L^\lambda\gr^\conv (f) &= \lfloor\mu_X\circ P\lambda_X\circ PF(\chi_{\gr^\conv(f)})\circ \eta_{FX}\circ Ff\rfloor\\
&= \lfloor\lambda_X\circ F(\chi_{\gr^\conv(f)})\circ Ff\rfloor\\
&= \lfloor\lambda_X\circ F(\chi_{\gr^\conv(f)}\circ f)\rfloor\\
&\overset{*}{\geq} \lfloor\lambda_X\circ F\eta_X\rfloor\\
&\overset{**}{\geq} \lfloor\eta_{FX}\rfloor = \Delta_{FX}
\end{align*}
where in inequality (*) we use monotonicity of $\lambda$, together with the fact that $\chi_{\gr^\conv(f)}\circ f \geq \eta_X$; and inequality (**) is simply the unit part of lax Eilenberg-Moore. Since $\gr^\conv(Ff)$ is the least relation $R$ with $\gr(Ff);R \geq \Delta_X$, we obtain
\[
L^\lambda\gr^\conv(f) \geq \gr^\conv(Ff)
\]
as desired.
\item[(diagonal-preserving)] Assume that $\lambda$ is laxly extensional. Then
\[
L^\lambda\Delta_Z = \lfloor\lambda_Z\circ\eta_Z\rfloor \leq \lfloor\eta_{FZ}\rfloor = \Delta_{FZ}
\]
\item[(symmetry)] Assume that $\lambda$ is symmetric. Then it follows immediately that
\[
L^\lambda(R^\conv) = \lfloor\lambda_X\circ F(\chi_{R}^\flat)\rfloor = \lfloor(\lambda_Y\circ F\chi_R)^\flat\rfloor = \lfloor \lambda_Y\circ F\chi_R\rfloor^\conv = (L^\lambda R)^\conv
\]
\end{description}
\end{enumerate}
\end{proof}

\section{Explicit descriptions}

Since the class of $F$-liftings forms a complete lattice for each $F$, it follows that each $F$ has a minimal lifting $\tilde F$. In the case of weak-pullback preserving $F$, we know that $\tilde F = \bar F$, the Barr lifting. However, for non-weak-pullback preserving functors, giving an explicit description of the minimal lifting involves a non-trivial amount of effort. 

In this section, we will study the minimal liftings for the neighborhood functor and the monotone neighborhood functor. For the (ordinary) neighborhood functor, we moreover give a full description of the complete lattice of liftings.

\subsection{Monotone neighborhood functor}

Recall the lifting $\widetilde \M$ from example \ref{ex:monotonelift}.

\begin{theorem}\label{thm:monotonelift}
The lifting $\widetilde \M$ is the minimal lifting for the monotone neighborhood functor $\M$.
\end{theorem}

To prove this, we first need a lemma.

\begin{lemma}\label{lemma:totsur}
Let $R:X\relto Y$ be a total surjective relation. Then $\widetilde \M R\leq LR$ for all liftings $L$.
\end{lemma}

In \cite{Hansen2004}, a similar statement appears as lemma 4.7. 

\begin{proof}
Consider the two projection morphisms $\pi_X:R\to X$ and $\pi_Y:R\to Y$. Since $R$ is total and surjective, both these functions are surjective. 

We claim that $\widetilde\M R = (\M \pi_X)^\circ;\M\pi_Y$. The inequality $\geq$ follows from $R = (\pi_X)^\circ;\pi_Y$. 

For $\leq$, let $(U,V)\in \widetilde\M R$. Then we set
\begin{align*}
W_0 &:= \{\{(x,y)\in R\mid x\in u\}\mid u\in U\}\\
W_1 &:= \{\{(x,y)\in R\mid y\in V\}\mid v\in V\}\\
W &:= \{w\mid \exists w'\in W_0\cup W_1:w'\subseteq w\}
\end{align*}
We claim that $\M\pi_X(W) = U$. For this, we need to show that (1) if $u\in U$, then $\pi_X^{-1}(u)\in W$, and (2) if $\pi_X^{-1}(u)\in W$, then $u\in U$.
\begin{enumerate}[label = (\arabic*)]
\item Clearly, if $u\in U$, then $\pi_X^{-1}(u) = \{(x,y)\in R\mid x\in u\}\in W$, so $\pi_X^{-1}(u)\in W$. 
\item Assume $\pi_X^{-1}(u)\in W$. There are two cases: (i) there is a $u'\in U$ with $\{(x,y)\in R\mid x\in u'\}\subseteq \pi_X^{-1}(u)$, or (ii) there is a $v\in V$ with $\{(x,y)\in R\mid y\in v'\}\subseteq \pi_X^{-1}(u)$. 
\begin{enumerate}[label = (\roman*)]
\item In this case, we know that $\pi_X[\{(x,y)\in R\mid x\in u'\}]\subseteq \pi_X(\pi_X^{-1}(u))$. But since $R$ was total, we know that $\pi_X[\{(x,y)\in R\mid x\in u'\}] = u'$ and $\pi_X[\pi^{-1}(u)] = u$. So $u'\subseteq u$, and hence $u\in U$.
\item Clearly, $\pi_X[\{(x,y)\in R\mid y\in v\}] = \{x\mid \exists y\in v:xRy\}$. Since $(U,V)\in \widetilde\M R$, there is a $u'\in U$ such that for all $x\in u'$, there is a $y\in v$ with $xRy$. But this just says that $u'\subseteq \pi_X[\{(x,y)\in R\mid y\in v\}]$. So we conclude that there is a $u'\in U$ with 
\[u'\subseteq \pi_X[\{(x,y)\in R\mid y\in v\}] \subseteq \pi_X(\pi_X^{-1}(u)) = u\]
and hence $u\in U$.
\end{enumerate}
So in both cases, we have $u\in U$, as desired.
\end{enumerate}
The proof that $\M\pi_Y(W) = V$ is completely symmetrical; so, we can conclude that $(U,V)\in (\M \pi_X)^\circ;\M\pi_Y$. 

\vspace{10pt}

Now, let $L$ be any lifting. Then
\[
LR = L((\pi_X)^\circ;\pi_Y)\geq L(\pi_X)^\circ;L\pi_Y\geq (\M \pi_X)^\circ;\M\pi_Y = \widetilde\M R
\]
\end{proof}

With this lemma, we can prove theorem \ref{thm:monotonelift}.

\begin{proof}
Let $R:X\relto Y$ be any relation. Let $X'$ be the domain of $R$ and $Y'$ the range of $R$. Then we define $X_* = X\cup\{*\}, Y_* = Y\cup \{*\}$ and
\[
R_* = R\cup \{(x,*)\mid x\in X\setminus X'\}\cup \{(*,y)\mid y\in Y\setminus Y'\}\cup \{(*,*)\}
\]
Then $R_*:X_*\relto Y_*$ is total and surjective. 

\vspace{10pt}

Let $\iota_X:X\to X_*$ and $\iota_Y:Y\to Y_*$ be the natural inclusion functions. First, we note that
$
R = \iota_X;R_*;(\iota_Y)^\circ
$
The inequality $\leq$ is clear, since $R\subseteq R_*$. For $\geq$, notice that $*$ is not in the range of either $\iota_X$ or $\iota_Y$. 

Now by lemma \ref{lemma:cospan}, we know that for any lifting $L$,
\[
LR = (\M\iota_X);LR_*;(\M\iota_Y)^\circ.
\]
So we can calculate that
\begin{align*}
LR&= \M\iota_X;LR_*;(\M\iota_Y)^\circ\\
&\geq \M\iota_X;\widetilde\M R_*;(\M\iota_Y)^\circ\qquad \text{ by lemma \ref{lemma:totsur}}\\
&= \widetilde\M R
\end{align*}
We conclude that $\widetilde\M$ is minimal. 
\end{proof}

\subsection{The neighborhood functor}

We introduce an extremely minimal logic for neighborhood systems. This will consist of the following expressions:
\begin{align*}
\rho_0 &::= \square \bot\mid\neg\square\bot\\
\rho_1 &::= \square\top \mid \neg \square \top\\
\rho &::= (\rho_0, \rho_1)
\end{align*}
Given $(U,V)\in \N X \times \N Y$, satisfaction $(U,V)\Vdash \rho$ is defined as follows:
\begin{align*}
(U,V)&\Vdash \square \bot\text{ iff } \varnothing \in U\implies \varnothing\in V\\
(U,V)&\Vdash \neg\square\bot \text{ iff }\varnothing \notin U \implies \varnothing\notin V\\
(U,V)&\Vdash \square\top \text{ iff }X\in U\implies Y\in V\\
(U,V)&\Vdash \neg\square\top \text{ iff }X\notin U\implies Y\notin V\\
(U,V)&\Vdash (\rho_0, \rho_1) \text{ iff }(U,V)\Vdash \rho_0\text{ and }(U,V)\Vdash \rho_1
\end{align*}
Now let $I$ be the set of all $\rho$'s. For each $J\subseteq I$, we get a lifting $L_J$ defined on a relation $R:X\relto Y$ as
\[
L_J(R) := \{(U,V)\in \N X\times \N Y\mid (U,V)\Vdash \rho\text{ for all }\rho\in J\}
\]
Since these liftings do not depend on the chosen relation, we will omit $R$, writing simply $L_J: \N X \relto \N Y$. Note also that $J\supseteq J'$ if and only if $L_J\leq L_{J'}$.

\begin{theorem}
The lattice $(P(I),\supseteq)$ is isomorphic to $(\Lift(\N), \leq)$ via $J\mapsto L_J$. 
\end{theorem}

To prove this theorem, we will need the following lemma:
\begin{lemma}\label{lemma:nbdlogic}
Let $(U,V)\in \N X \times \N Y$ and $(U', V')\in \N X'\times \N Y'$. Assume that for some $\rho$, we have $(U,V)\nVdash \rho$ and $(U',V')\nVdash \rho$. Then for each $\rho'$, we have
\[
(U,U')\Vdash \rho', \quad (V,V')\Vdash \rho'
\]
\end{lemma}

\begin{proof}
WLOG, we can assume that $\rho = (\square \bot, \square \top)$; all other cases are similar.

Then since $(U,V)\nVdash \rho$, we know $\varnothing\in U, \varnothing\notin V$ and $X\in U, Y\notin V$. Similarly, we know $\varnothing\in U', \varnothing\notin V$ and $X'\in U', Y'\notin V'$. But from these data, it follows immediately that for all $\rho'$, we must have
\[
(U, U')\Vdash \rho'
\]
since $U$ and $U'$ agree on $\varnothing$ and the entire set. And of course the same holds for $(V, V')$. 

\end{proof}

Now we can start the full proof.

\begin{proof}
First, we show that each $L_J$ is a lifting. Since clearly $L_J = \bigwedge_{\rho\in J}L_{\{\rho\}}$, it suffices to show that each $L_{\{\rho\}}$ is a lifting.

They are clearly monotonic, since they do not depend on the input $R$. They are also clearly laxly functorial. Finally, if $f:X\to Y$ is a function, then for all $U\in \N X$ and all $\rho\in I$, we have
\[
(U, (\N f)U)\Vdash \rho
\]
since
\[
(\N f)U\ni \varnothing \text{ iff } U\ni f^{-1}(\varnothing)\text{ iff }U\ni \varnothing\text{ and }
(\N f)\ni X \text{ iff }U\ni f^{-1}(X) \text{ iff }U\ni Y
\]
So indeed, each $L_J$ extends the graph of $\N f$. 

\vspace{10pt}

This shows that the map $J\mapsto L_J$ is well-defined. It is clearly injective and meet-preserving (recall that the meet in $(P(I), \supseteq)$ is given by union), so it remains to show that it is surjective. We will proceed in three steps:
\begin{enumerate}
\item The top element is preserved by $J\mapsto L_J$;
\item The bottom element is preserved by $J\mapsto L_J$;
\item If $L > L_J$, then there is some $J' \subsetneq J$ with $ L \geq L_{J'}$. 
\end{enumerate} 
These three steps together imply that $J\mapsto L_J$ is surjective, from which it then follows that it is an isomorphism.

\vspace{10pt}

For point 1: The top element of $(P(I), \supseteq)$ is $\varnothing$, and indeed $L_{\varnothing}(R:X\relto Y) = X\times Y$. 

For point 2: Let $L$ be any symmetric lifting for $\N$. For a given $X$, write $0_X:X\relto X$ for the empty relation. We will show that $(U,V)\in L0_X$ if $U$ and $V$ agree on $\varnothing$ and $X$. 

\vspace{10pt}

We first assume that $X$ contains some point $x_0$. Write $2 = \{a,b\}$ for the generic two-element set; by abuse of notation, we may also consider $a,b:X\to 2$ and $x_0:X\to X$ as constant maps.

Let $U\in \N X$ be a neighborhood system. There are four cases:
\begin{enumerate}[label = (\roman*)]
\item $\varnothing\notin U, X\notin U$. Then we see that
\[
\N a(U) = \varnothing = \N b(U)
\]
since for constant maps $c:X\to Y$, we have $c^{-1}(A) = \varnothing$ or $c^{-1}(A) = X$ for all $A$. We also clearly have $\N b(\varnothing) = \varnothing$. So, we have
\[
(U, \varnothing)\in \gr(\N a), (\varnothing, U) \in \gr^\conv(\N b)
\]
for all $U$ omitting $\varnothing$ and $X$. Now we have if $U, V$ both omit $\varnothing$ and $X$, then
\[
U (L\gr(a)) \varnothing (L\gr^\conv(b)) V
\]
and hence
\[
(U, V)\in L\gr(a);L\gr^\conv(b) \subseteq L0_X
\]
\item $\varnothing \notin U, X\in U$. Then $\N a(U) = \{\varnothing, \{b\}\}$. Take $f:X \to 2$ given by
\[
f(x) = \begin{cases} b & x \neq x_0\\a & x = x_0\end{cases}
\]
Let $V = \{\varnothing, X\setminus\{x_0\}\}$. Then it is easily seen that $\N f(V) = \{\varnothing, \{b\}\}$. Finally, we have $\N a(V) = \N b(V) = \{\varnothing\}$, again by the remarks on inverse images along constant maps. 

Now we have a `zigzag' as in figure \ref{fig:a}.
\begin{figure}
\begin{subfigure}{0.5\textwidth}
\begin{tikzpicture}[shorten <= 15pt, shorten >=15pt]
\node at (0,2) {$U : X$};
\node at (2,0) {$\{\varnothing, \{b\}\} : 2$};
\node at (4,2) {$\{\varnothing, X\setminus\{a\}\} : X$};
\node at (6,0) {$\{\varnothing\} : 2$};

\draw[->] (0,2) -- node[anchor = west] {$a$} (2,0);
\draw[->] (4,2) -- node[anchor = west] {$f$} (2,0);
\draw[->] (4,2) -- node[anchor = west] {$a$ or $b$} (6,0);
\end{tikzpicture}
\caption{\label{fig:a}}
\end{subfigure}
\begin{subfigure}{0.5\textwidth}
\begin{tikzpicture}[shorten <= 15pt, shorten >=15pt]
\node at (0,2) {$U : X$};
\node at (2,0) {$\{\{a\}, \{a,b\}\} : 2$};
\node at (4,2) {$\{\{a\}, X\} : X$};
\node at (6,0) {$\{\{a,b\}\} : 2$};

\draw[->] (0,2) -- node[anchor = west] {$a$} (2,0);
\draw[->] (4,2) -- node[anchor = west] {$f$} (2,0);
\draw[->] (4,2) -- node[anchor = west] {$a$ or $b$} (6,0);
\end{tikzpicture}
\caption{\label{fig:b}}
\end{subfigure}
\end{figure}
By tracing the definitions, we can see that $\gr(a);\gr^\conv(f) = \gr(x_0)$, and
\[
\gr(a);\gr^\conv(f);\gr(a) = \gr(x_0);\gr(a) = \gr(a)
\]
and similarly $\gr(a);\gr^\conv(f);\gr(b) = \gr(b)$. We now have that if $U$ is such that $\varnothing\in U, X\notin U$, then
\[
(U, \{\varnothing\})\in L(\gr(a)),\text{ and }(\{\varnothing\}, U)\in L(\gr^\conv(b))
\]
But now for all $U, V$ which both contain $\varnothing$ and both omit $X$, we have
\[
(U,V) \in L\gr(a);L(\gr^\conv(b))\subseteq L(\gr(a);\gr^\conv(b)) = L0_X
\]
\item Let $U$ be such that $\varnothing\notin U, X\in U$. Then $\N a(U) = \{\{a\}, \{a,b\}\}$. Take $V = \{\{x_0\}, X\}$. Then with $f$ as in point (ii), we have $\N f(V) = \{\{a\}, \{a,b\}\}$. Now for the constant maps $a, b$, we have $\N a(V) = \{\{a,b\}\} = \N b(V)$. Hence, we obtain a similar zigzag as in point (ii), as can be seen in figure \ref{fig:b}. From here, the argument is completely the same as in (ii): for $U, V$ both omitting $\varnothing$ and both including $X$, we get
\[
(U,\{\{a,b\}\})\in L\gr(a),\qquad  (\{\{a,b\}\},V)\in L\gr^\conv(b)
\]
showing that
\[
(U,V)\in L(\gr(a);(\gr(b))^\conv) = L0_X
\]
\item $\varnothing \in U, X\in U$. Then $\N a(U) = P2 = \N b(U)$, and so as in (i) we get for all $U, V$ both including $\varnothing$ and $X$ that
\[
(U, P2)\in L\gr(a),\qquad (P2, V)\in L\gr^\conv(b)
\]
and hence
\[
(U, V) \in L(\gr(a);\gr^\conv(b)) = L0_X
\]
\end{enumerate}

Now we have that if $X$ is nonempty, then $L0_X \supseteq L_I$. But of course, if $X$ and $Y$ are arbitrary, then the empty relation $0_{XY}:X\relto Y$ factors through $0_{X+Y}$ via the inclusions $\iota_X:X\to X+Y, \iota_Y:Y\to X+Y$. From this, it follows easily that $L0_{XY}\supseteq L_I$. But now, for $R:X\relto Y$ an arbitrary relation, we have that
\[
LR\supseteq L0_{XY}\supseteq L_I
\]
showing that $L_I$ is minimal indeed.

\vspace{10pt}

For point 3: Let $L$ be any lifting, and $J\subseteq I$ with $L > L_J$. Then there is some relation $R:X\relto Y$ and some neighborhood systems $(U,V)\in \N X\times \N Y$ with $(U,V)\in LR$ and $(U,V)\nVdash \rho_0$ for some $\rho_0\in J$. 

We claim that now for $J' = J\setminus \{\rho_0\}$, we have $L\geq L_{J'}$. Again, we will show that $L0_{X'Y'} \geq L_{J'}$ for all $X', Y'$.

Now let $(U',V')\in L_{J'}$. There are two cases:
\begin{enumerate}[label = (\roman*)]
\item $(U',V')\Vdash \rho_0$. Then $(U', V')\in L_J < L$, so $(U', V')\in L0_{X'Y'}$.
\item $(U',V') \nVdash \rho_0$. Since $(U,V)\nVdash \rho_0$ we know by lemma \ref{lemma:nbdlogic} that
\[
(U', U)\in L_I,\quad (V,V')\in L_I
\]
and hence
\begin{align*}
(U', V')&\in L_I;LR;L_I\subseteq L0_{X'X};LR;L_{YY'}\subseteq L(0_{X'X};R;0_{YY'})= L0_{X'Y'}
\end{align*}
So indeed, $L0_{X'Y'}\geq L_{J'}$ and hence for arbitrary relations $R':X'\relto Y'$ we have
\[
LR\geq L0_{X'Y'}\geq L_{J'}
\]
as desired.
\end{enumerate}

\end{proof}

\section{Conclusion and further research}

We have shown that for a fixed functor, the lax liftings form a complete lattice. In particular, any functor admits a minimal, ``maximally expressive" lifting. We have shown that for weak pullback-preserving functors, the least functor coincides with the Barr lifting. 

We have defined lax distributive laws, and shown that there is an isomorphism between the lattice of lax liftings, and the lattice of lax distributive laws. We also characterized those distributive laws that correspond to liftings that are symmetric and 
diagonal-preserving. 

We studied the monotone and ordinary neighbourhood functors in more detail. For the monotone neighbourhood functor, we have shown that the known lifting $\widetilde \M$ is minimal. For the ordinary neighbourhood functor, we have explicitly described all 16 liftings. This question was still open in \cite{Schoen2021}.

The results in this paper are specific to the categories of $\Sets$ and 2-valued relations. Other kinds of liftings have been considered. For instance, in \cite{Wild2020}, liftings of fuzzy relations are defined. A natural direction of further research is to investigate if the results from this paper could be extended to cover a wider range of many-valued relations. More generally still, one can see $\Sets$ as the category of functions inside the allegory $\Rel$. A possible approach would be to study liftings in the setting of arbitrary (power) allegories. 

\newpage

\bibliographystyle{eptcs}
\bibliography{laxliftingreferences}

\appendix

\section{Additional proofs}

\begin{proof}[Proof of theorem \ref{thm:meets}]
We show that $\Lift(F)$ has all meets. Let $\{L_i\mid i\in I\}$ be any collection of $F$-liftings. For a given $R:X\relto Y$, set
\[
LR = \bigcap_{i\in I} L_iR
\]
We show that $L$ is again a lifting, by showing it satisfies conditions 1, 2 and 3.
\begin{enumerate}[label = (\arabic*)]
\item If $R\leq S$, then 
\[
LR = \bigcap_{i\in I}L_iR\leq \bigcap_{i\in I}L_iS = LS
\]
\item If $R:X\relto Y$ and $S:Y\relto Z$ are relations, then
\begin{align*}
LR;LS &= \left(\bigcap_{i\in I}L_iR\right);\left(\bigcap_{i\in I}L_iS\right)\\
&\leq \bigcap_{i\in I} \bigcap_{j\in I}L_iR;L_jS\\
&\leq \bigcap_{i\in I}L_iR;L_iS\\
&\leq \bigcap_{i\in I}L_i(R;S)\\
&= L(R;S)
\end{align*}
\item If $f:X\to Y$ is a function, then
\[
L\gr(f) = \bigcap_{i\in I}L_i\gr(f) \geq \bigcap_{i\in I}\gr(Ff) = \gr(Ff).
\]
The other inequality is similar. 
\end{enumerate}
So $L$ is a lifting, and is clearly the greatest lower bound for the $L_i$. 

\end{proof}

\begin{proof}[Proof of theorem \ref{thm:transform}]
\begin{enumerate}[label = (\roman*)]
\item We check the three conditions.
\begin{description}
\item[(2-cells)] If $R\leq R'$, then
\[
\eta^*L(R) = (\eta \times\eta)^{-1}(LR) \leq ( \eta \times \eta)^{-1}(LR') = \eta^*L(R')
\]
since for any function $f$, we know that $f^{-1}$ preserves inclusions. 
\item[(lax functoriality)] If $R:X\relto Y$ and $S:Y\relto Z$, we have
\begin{align*}
\eta^*L(R;S) &= \{(x,z)\mid (\eta(x),\eta(z))\in L(R;S)\}\\
&\geq \{(x,z)\mid \eta(x,z)\in LR;LS\}\\
&\geq \{(x,z)\mid \exists y\in Y:(\eta(x),\eta(y))\in LR, (\eta(y),\eta(z))\in LS\}\\
&= \eta^*L(R);\eta^*L(S)
\end{align*}
\item[(lifting)] Let $f:X\to Y$ be a function. Naturality of $\eta$ states that $\gr(Ff);\gr(\eta) = \gr(\eta);\gr(Gf)$. From this, it follows that $\gr(Ff) \leq \gr(\eta);\gr(Gf);\gr^\conv(\eta)$. Hence, we have
\[
\gr(Ff) \leq \gr(\eta);\gr(Gf);\gr^\conv(\eta) \leq \gr(\eta);L\gr(f);\gr^\conv(\eta) = \eta^*L(\gr(f))
\]
and
\[
\gr^\conv(Ff) \leq \gr(\eta);\gr^\conv(Gf);\gr^\conv(\eta) \leq \gr(\eta);L(\gr^\conv(f));\gr^\conv(\eta) = \eta^*L(\gr^\conv(f))
\]
\end{description}
\item For meets, we have
\[
\eta^*(\bigwedge_i L_i)(R) =(\eta\times\eta)^{-1}(\bigcap_i (L_iR)) = \bigcap_i(\eta\times\eta)^{-1}(L_iR) = \left(\bigwedge_iL_i\right)(R)
\]
since meets are preserved by inverse images. For $\twiddle{(-)}$, we have
\begin{align*}
\eta^*(\twiddle L)(R) &= (\eta\times\eta)^{-1}(\twiddle LR)\\
&= (\eta\times\eta)^{-1}((L(R^\conv))^\conv)\\
&= ((\eta\times\eta)^{-1}(L(R^\conv)))^\conv\\
&= (\eta^*L(R^\conv))^\conv\\
&= \twiddle{(\eta^*L)}(R)
\end{align*}
\item This follows directly from preservation of $\twiddle{(-)}$: we have
\begin{align*}
L\text{ is symmetric }&\iff L = \twiddle L\\
&\implies \eta^*L = \eta^*(\twiddle L)\\
&\iff \eta^*L = \twiddle(\eta^*L)\\
&\iff \eta^*L \text{ is symmetric}
\end{align*}
\item Assume $\eta$ is everywhere injective, and $L$ preserves diagonals. Then let $X$ be arbitrary. For all $(x,y)\in FX \times FX$, we have
\begin{align*}
(x,y)\in \eta^*L\Delta_X&\iff(\eta(x),\eta(y))\in L\Delta_X\\
&\implies \eta(x) = \eta(y) & \text{ since }L\text{ preserves diagonals}\\
&\implies x = y & \text{ since }\eta\text{ is injective}
\end{align*}
and hence $\eta^*L$ preserves diagonals. 
\end{enumerate}
\end{proof}

\end{document}